\documentclass[oneside,a4paper, 12pt]{article}
\usepackage[left=2.5cm,right=2.5cm,top=3cm,bottom=3cm]{geometry}
\usepackage{amsmath, amssymb, amsfonts, mathtools}
\usepackage{amsthm}
\usepackage{setspace}
\usepackage{graphicx}       
\usepackage[bottom]{footmisc}

\usepackage{hyperref}
    \hypersetup{
    colorlinks=true,
    linkcolor=black,
    filecolor=magenta,
    urlcolor=black,
    citecolor=black,
    linktocpage=false,
    breaklinks=true,
    }

\usepackage{enumitem}

\theoremstyle{plain}
\newtheorem{theorem}{Theorem}

\newtheorem{lemma}[theorem]{Lemma}
\newtheorem{corollary}[theorem]{Corollary}

\theoremstyle{definition}

\newtheorem{example}{Example}

\newtheorem{remark}{Remark}

\title{\bfseries
    Pathwise approximations for the solution of the non-linear
    filtering problem$^\star$
    \footnotetext{
        $^\star$ Submitted for publication in
        \emph{
        Stochastic Analysis, Filtering, and Stochastic Optimization:
        A Commemorative Volume to Honor Mark H. A. Davis's Contributions
        },
        Springer Verlag.
    }\bigskip
}

\author{\centering
    \begin{minipage}[t]{0.3\linewidth}
    \begin{spacing}{0.7}
    {\small\bfseries Dan Crisan}\\
    {\footnotesize
    Imperial College London\\
    Department of Mathematics\\
    Huxley's Building\\
    180 Queen's Gate\\
    London SW7 2AZ, UK\\
   \href{mailto:d.crisan@imperial.ac.uk}{\texttt{d.crisan@imperial.ac.uk}}
    }
    \end{spacing}
    \end{minipage}
    \begin{minipage}[t]{0.3\linewidth}
    \begin{spacing}{0.7}
    {\small\bfseries Alexander Lobbe}\\
    {\footnotesize
        University of Oslo\\
        Department of Mathematics\\
        P.O.~Box 1053, Blindern\\
        0316 Oslo, Norway\\
        \href{mailto:alexalob@math.uio.no}{\texttt{alexalob@math.uio.no}}
    }
    \end{spacing}
    \end{minipage}
    \begin{minipage}[t]{0.3\linewidth}
    \begin{spacing}{0.7}
    {\small\bfseries Salvador Ortiz-Latorre}\\
    {\footnotesize
        University of Oslo\\
        Department of Mathematics\\
        P.O.~Box 1053, Blindern\\
        0316 Oslo, Norway\\
        \href{mailto:salvadoo@math.uio.no}{\texttt{salvadoo@math.uio.no}}
    }
    \end{spacing}
    \end{minipage}
}
\date{}
\begin{document}
\maketitle

\abstract{\normalsize
We consider high order approximations of the solution of the stochastic
filtering problem, derive their pathwise representation in the spirit of
the earlier work of Clark~\cite{clark1978design} and
Davis~\cite{davis1982pathwise, davis1987nonlinear} and prove their
robustness property.
In particular, we show that the high order discretised filtering
functionals can be represented by Lipschitz continuous functions
defined on the observation path space.
This property is important from the practical point of view as it is in fact
the pathwise version of the filtering functional that is sought in numerical
applications. Moreover, the pathwise viewpoint will be a stepping stone into the
rigorous development of machine learning methods for the filtering problem.
This work is a continuation of~\cite{crisan2019high} where a discretisation of
the solution of the filtering problem of arbitrary order has been established.
We expand the work in~\cite{crisan2019high} by showing that robust
approximations can be derived from the discretisations therein.}

\section{Introduction}
\label{sec:intro}
With the present article on non-linear filtering we wish to honor the work
of Mark H.~A.~Davis in particular to commemorate our great colleague.
The topic of filtering is an area that has
seen many excellent contributions by Mark. It is remarkable that he was able
to advance the understanding of non-linear filtering from a variety of angles.
He considered many aspects of the field in his work, spanning the full
range from the theory of the filtering equations to the numerical solution of
the filtering problem via Monte-Carlo methods.

Mark Davis' work on filtering can be traced back to his doctoral thesis where
he treats stochastic control of partially observable processes.
The first article specifically on the topic of filtering
that was co-authored by Mark appeared back in
1975 and considered a filtering problem with discontinuous observation
process~\cite{davis1975nonlinear}. There, they used the so-called innovations
method to compute the evolution of the conditional density of a process that is
used to modulate the rate of a counting process. This method is nowadays
well-known and is a standard way also to compute the linear (Kalman) filter
explicitly.
Early on in his career, Mark also contributed to the dissemination of filtering
in the mathematics community with his monograph
\emph{Linear Estimation and Stochastic Control}~\cite{davis1977linear},
published in 1977, which deals with filtering to a significant degree.
Moreover, his paper
\emph{An Introduction to Nonlinear Filtering}~\cite{davis1981introduction},
written together with S.~I. Marcus in 1981, has gained the status of a standard
reference in the field.

Importantly, and in connection to the theme of the present paper, Mark has
worked on computation and the robust filter already in
1980~\cite{davis1980computational}.
Directly after the conception of the robust filter by Clark in
1978~\cite{clark1978design}, Mark took up the role of a
driving figure in the subsequent development of robust, also known as pathwise,
filtering theory~\cite{davis1982pathwise, davis1987nonlinear}.
Here, he was instrumental in the development of the pathwise solution to the
filtering equations with one-dimensional observation processes.
Additionally, also correlated noise was already analysed in this work.

Robust filtering remains a highly relevant and challenging problem today.
Some more recent work on this topic includes the article~\cite{crisan2013robust}
which can be seen as an extension of the work by Mark, where correlated noise
and a multidimensional observation process are considered.
The work~\cite{crisan2005version} is also worth mentioning in this context,
as it establishes the validity of the robust filter rigorously.

Non-linear filtering is an important area within stochastic analysis and has
numerous applications in a variety of different fields.
For example, numerical weather prediction requires the solution of a high
dimensional, non-linear filtering problem.
Therefore, accurate and fast numerical algorithms for the approximate solution
of the filtering problem are essential.
In this contribution we analyse a recently developed high order time
discretisation of the solution of the filtering problem from the
literature~\cite{crisan2019high} and prove that the so discretised solution
possesses a property known as \emph{robustness}. Thus, the present paper is a
continuation of the previous work~\cite{crisan2019high} by two of the authors
which gives a new high-order time discretisation for the filtering functional.
We extend this result to produce the robust version, of any order, of the
discretisation from~\cite{crisan2019high}.
The implementation of the resulting numerical method remains open and is subject
of future research.
In subsequent work, the authors plan to deal with suitable extensions, notably a
machine learning approach to pathwise filtering.

Robustness is a property that is especially important for the numerical
approximation of the filtering problem in continuous time, since numerical
observations can only be made in a discrete way. Here, the robustness property
ensures that despite the discrete approximation, the solution obtained from it
will still be a reasonable approximation of the true, continuous filter.

The present paper is organised as follows: In Section~\ref{sec:prelim} we
discuss the established theory leading up to the contribution of this paper.
We introduce the stochastic filtering problem in sufficient generality in
Subsection~\ref{sec:filtering} whereafter the high order discretisation from the
recent paper~\cite{crisan2019high} is presented in
Subsection~\ref{sec:discretisation} together with all the necessary notations.
The Subsection~\ref{sec:discretisation} is concluded with the
Theorem~\ref{thm: Main Filtering_2}, taken from \cite{crisan2019high},
which shows the validity of the high order discretisation and is the
starting point for our contribution.
Then, Section~\ref{sec:robustness} serves to concisely present the main result
of this work, which is Theorem~\ref{thm:robust} below. Our Theorem is a
general result applying to corresponding discretisations of arbitrary order and
shows that all of these discretisations do indeed assume a robust version.
In Section~\ref{sec:proof} we present the proof of the main result in detail.
The argument proceeds along the following lines.
First, we establish the robust version of the discretisations for any order by
means of a \emph{formal} application of the integration by parts formula.
In Lemma~\ref{lem:Z_bound} we then show that the new robust approximation is
locally bounded over the set of observation paths.
Thereafter, Lemma~\ref{lem:G_lipschitz} shows that the robustly discretised
filtering functionals are locally Lipschitz continuous over the set of
observation paths. Based on the elementary but important auxilliary
Lemma~\ref{lem:H13} we use the path properties of the typical observation
in Lemma~\ref{lem:int_version} to get a version of the stochastic integral
appearing in the robust approximation which
is product measurable on the Borel sigma-algebra of the path space and the
chosen filtration.
Finally, after simplifying the arguments by lifting some of the
random variables to an auxilliary copy of the probability space, we can show in
Lemma~\ref{lem:null_set} that, up to a null-set, the lifted stochastic integral
appearing in the robust approximation is a random variable on the correct space.
And subsequently, in Lemma~\ref{lem:integral_representation} that the pathwise
integral almost surely coincides with the standard stochastic integral of the
observation process. The argument is concluded with Theorem~\ref{thm:final}
where we show that the robustly discretised filtering functional is a version of
the high-order discretisation of the filtering functional as derived in the
recent paper~\cite{crisan2019high}.

Our result in Theorem~\ref{thm:robust} can be interpreted as a remedy for some
of the shortcomings of the earlier work~\cite{crisan2019high} where the
discretisation of the filter is viewed as a random variable and the dependence
on the observation path is not made explicit.
Here, we are correcting this in the sense that we give an interpretation of said
random variable as a continuous function on path space.
Our approach has two main advantages. Firstly, from a
practitioner's point of view, it is exactly the path dependent version of the
discretised solution that we are computing in numerical applications. Thus it is
natural to consider it explicitly.
The second advantage lies in the fact that here we are building a foundation for
the theoretical development of machine learning approaches to the filtering
problem which rely on the simulation of observation paths.
With Theorem~\ref{thm:robust} we offer a first theoretical justification for
this approach.

\section{Preliminaries}
\label{sec:prelim}
Here, we begin by introducing the theory leading up to the main part of the
paper which is presented in Sections~\ref{sec:robustness} and \ref{sec:proof}.
\subsection{The filtering problem}
\label{sec:filtering}
Let $\left(\Omega,\mathcal{F},P\right)$ be a probability space
with a complete and right-continuous filtration $(\mathcal{F}_{t})_{t\geq0}$.
We consider a $d_{X}\times d_{Y}$-dimensional partially observed system
$\left(X,Y\right)$ satisfying the system of stochastic integral equations
\begin{equation}
\label{eq:system}
\left\{
    \begin{aligned}
    X_{t} & =X_{0} + \int_{0}^{t} f\left(X_{s}\right) ds
                   + \int_{0}^{t} \sigma\left(X_{s}\right) dV_{s},\\
    Y_{t} & =        \int_{0}^{t} h\left(X_{s}\right)ds
                   + W_{t},
    \end{aligned}
    \right.
\end{equation}
where $V$ and $W$ are independent $(\mathcal{F}_{t})_{t\geq0}$-adapted $d_{V}$-
and $d_{Y}$-dimensional standard Brownian motions, respectively.
Further, $X_{0}$ is a random variable, independent of $V$ and $W$,
with distribution denoted by $\pi_{0}$.
We assume that the coefficients
\begin{equation*}
    f=\left(f_{i}\right)_{i=1,\ldots,d_{X}}:
        \mathrm{R}^{d_{X}}\rightarrow\mathrm{R}^{d_{X}}
\text{ and }
    \sigma=\left(\sigma_{i,j}\right)_{i=1,\ldots,d_{X},j=1,\ldots,d_{V}}:
        \mathrm{R}^{d_{X}}\rightarrow\mathrm{R}^{d_{X}\times d_{V}}
\end{equation*}
of the \emph{signal process} $X$ are globally Lipschitz continuous
and that the \emph{sensor function}
\begin{equation*}
    h=\left(h_{i}\right)_{i=1,\ldots,d_{Y}}:
        \mathrm{R}^{d_{X}}\rightarrow\mathrm{R}^{d_{Y}}
\end{equation*}
is Borel-measurable and has linear growth. These conditions ensure that strong
solutions to the system \eqref{eq:system} exist and are almost surely unique.
A central object in filtering theory is the \emph{observation filtration}
$\left\{ \mathcal{Y}_{t}\right\}{\!}_{t\geq0}$ that is defined as the
augmentation of the filtration generated by the \emph{observation process} $Y$,
so that
$\mathcal{Y}_{t}=\sigma\left(Y_{s},s\in\left[0,t\right]\right)\vee\mathcal{N}$,
where $\mathcal{N}$ are all $P$-null sets of $\mathcal{F}$.

In this context, non-linear filtering means that we are interested in
determining, for all $t>0$, the conditional law, called \emph{filter} and
denoted by $\pi_{t}$, of the signal $X$ at time $t$ given the information
accumulated from observing $Y$ on the interval $\left[0,t\right]$.
Furthermore, this is equivalent to knowing for every bounded and Borel
measurable function $\varphi$ and every $t>0$, the value of
\begin{equation*}
    \pi_t(\varphi) =
    \mathrm{E}\big[\varphi(X_t) \bigm\vert \mathcal{Y}_{t} \big].
\end{equation*}

A common approach to the non-linear filtering problem introduced above is via a
change of probability measure.
This approach is explained in detail in the
monograph~\cite{bain2009fundamentals}.
In summary, a probability measure $\tilde{P}$ is constructed that is absolutely
continuous with respect to $P$ and such that $Y$ becomes a $\tilde{P}$-Brownian
motion independent of $X$. Additionally, the law of $X$ remains unchanged under
$\tilde{P}$.
The Radon-Nikodym derivative of $\tilde{P}$ with respect to $P$ is further given
by the process $Z$ that is given, for all $t\geq 0$, by
\begin{equation*}
    Z_{t}=\exp\left(
       \sum_{i=1}^{d_{Y}} \int_{0}^{t} h_{i}\left(X_{s}\right) dY_{s}^{i}
     - \frac{1}{2}\sum_{i=1}^{d_{Y}} \int_{0}^{t} h_{i}^{2}\left(X_{s}\right) ds
        \right).
\end{equation*}
Note that $Z$ is an $(\mathcal{F}_{t})_{t\geq0}$-adapted martingale under
$\tilde{P}$.
This process is used in the definition of another, measure-valued process $\rho$
that is given, for all bounded and Borel measurable functions $\varphi$ and all
$t\geq 0$, by
\begin{equation}
\label{eq:rho}
    \rho_t(\varphi) =
    \tilde{\mathrm{E}}\big[ \varphi(X_t)Z_t \bigm\vert \mathcal{Y}_t \big],
\end{equation}
where we denote by $\tilde{\mathrm{E}}$ the expectation with respect to
$\tilde{P}$.
We call $\rho$ the \emph{unnormalised filter}, because it is related to the
probability measure-valued process $\pi$ through the Kallianpur-Striebel formula
establishing that for all bounded Borel measurable functions $\varphi$ and all
$t\geq 0$ we have $P$-almost surely that
\begin{align}
\label{eq:KS}
    \pi_{t}(\varphi)=
    \dfrac{\rho_{t}(\varphi)}{\rho_{t}(\boldsymbol{1})}
    & =\dfrac{\mathrm{\tilde{E}}\left[\varphi(X_{t})Z_t
    \vert \mathcal{Y}_{t}\right]}
    {\mathrm{\tilde{E}}\left[ Z_t \vert \mathcal{Y}_{t}\right]}
\end{align}
where $\boldsymbol{1}$ is the constant function. Hence, the denominator
$\rho_{t}(\boldsymbol{1})$ can be viewed as the normalising factor for $\pi_t$.

\subsection{High order time discretisation of the filter}
\label{sec:discretisation}
As shown by the Kallianpur-Striebel formula \eqref{eq:KS}, $\pi_t(\varphi)$ is
a ratio of two conditional expectations.
In the recent paper~\cite{crisan2019high} a high order time discretisation of
these conditional expectations was introduced which leads further to a high
order time discretisation of $\pi_t(\varphi)$.
The idea behind this discretisation is summarised as follows.

First, for the sake of compactness, we augment the observation process as
$\hat{Y}_t = (\hat{Y}^i_t)_{i=0}^{d_Y} = (t, Y^1_t, \ldots, Y^{d_Y}_t)$ for all
$t\geq 0$ and write
\begin{equation*}
    \hat{h} =
    \biggl(-\frac{1}{2}\sum_{i=1}^{d_Y} h_i^2, h_1, \ldots, h_{d_Y}\biggr).
\end{equation*}
Then, consider the \emph{log-likelihood} process
\begin{equation}
\label{eq:log-likelihood}
    \xi_t = \log(Z_t)
          = \sum_{i=0}^{d_{Y}}
          \int_{0}^{t} \hat{h}_{i}(X_{s})
          \,\mathrm{d}\hat{Y}_{s}^{i},
          \qquad t\geq0.
\end{equation}
Now, given a positive integer $m$, the order $m$ time discretisation is achieved
by a stochastic Taylor expansion up to order $m$ of the processes
$\big(\hat{h}_i(X_t)\big){}_{t\geq 0}$, $i = 0, \ldots, d_Y$ in
\eqref{eq:log-likelihood}.
Finally, we substitute the discretised log-likelihood back into the original
relationships \eqref{eq:rho} and the Kallianpur-Striebel
formula \eqref{eq:KS} to obtain a discretisation of the filtering functionals.
However, it is important to note that for the orders $m > 2$
an additional truncation procedure is needed, which we will make precise
shortly, after introducing the necessary notation for the
stochastic Taylor expansion.

\subsubsection{Stochastic Taylor expansions}
\label{sec:s_taylor}
Let
$\mathcal{M}= \bigl\{\alpha\in\{ 0,\ldots,d_{V}\}^l : l=0,1,\ldots \bigr\}$
be the set of all multi-indices with range $\{ 0,\ldots,d_{V}\}$,
where $\emptyset$ denotes the multi-index of length zero.
For
$\alpha=(\alpha_{1},...,\alpha_{k})\in \mathcal{M}$ we adopt the notation
$\lvert \alpha \rvert = k$ for its length,
$\lvert \alpha\rvert _{0} =\# \lbrace j : \alpha_{j}=0 \rbrace$ for the number
of zeros in $\alpha$, and
$\alpha- =(\alpha_{1},...,\alpha_{k-1})$ and
$-\alpha =(\alpha_{2},...,\alpha_{k})$, for the right and left truncations,
respectively.
By convention $\left|\emptyset\right|=0$ and $-\emptyset=\emptyset-=\emptyset$.
Given two multi-indices
$\alpha,\beta\in\mathcal{M}$
we denote their concatenation by $\alpha\ast\beta$.
For positive and non-zero integers $n$ and $m$,
we will also consider the subsets of multi-indices
\begin{align*}
    \mathcal{M}_{n,m} &
        =\left\{ \alpha\in\mathcal{M}
        :n\leq\left\vert \alpha\right\vert
        \leq m\right\} ,\text{ and}\\
    \mathcal{M}_{m} &
        =\mathcal{M}_{m,m}=
        \left\{ \alpha\in\mathcal{M}:
        \left\vert \alpha\right\vert =m\right\}.
\end{align*}

For brevity, and by slight abuse of notation, we augment the
Brownian motion $V$ and now write
$V=\left(V^{i}\right){}_{i=0}^{d_{V}} = (t,V^{1}_t,\ldots,V^{d_V}_t)$
for all $t\geq 0$.
We will consider the filtration $\{\mathcal{F}_{t}^{0,V}\}{}_{t\geq0}$
defined to be the usual augmentation of the filtration generated by the process
$V$ and initially enlarged with the random variable $X_{0}$.
Moreover, for fixed $t\geq 0$, we will also consider the filtration
$\{\mathcal{H}_{s}^{t}=\mathcal{F}_{s}^{0,V}\lor\mathcal{Y}_{t}\}{}_{s\leq t}$.
For all $\alpha\in\mathcal{M}$ and all suitably integrable
$\mathcal{H}_{s}^{t}$-adapted processes
$\gamma=\left\{ \gamma_{s}\right\}{\!}_{s\leq t}$ denote by
$I_{\alpha}\left(\gamma_{\cdot}\right){}_{s,t}$ the Itô iterated
integral given for all $s\leq t$ by
\begin{equation*}
    I_{\alpha}(\gamma_{\cdot}){}_{s,t}=
    \begin{cases}
        \displaystyle \gamma_{t}, & \text{if} \left|\alpha\right|=0\\
        \displaystyle\int_{s}^{t}I_{\alpha-}(\gamma_{\cdot}){}_{s,u}
        \,dV_{u}^{\alpha_{|\alpha|}},
            & \text{if}
             \left\vert \alpha\right\vert \geq1.
    \end{cases}
\end{equation*}
Based on the coefficient functions of the signal $X$, we introduce
the differential operators $L^{0}$ and $L^{r}$, $r=1,...,d_{V}$, defined for all
twice continuously differentiable functions
$g:\mathrm{R}^{d_{X}}\rightarrow\mathrm{R}$ by
\begin{align*}
    L^{0}g & =\sum_{k=1}^{d_{X}}f_{k}
        \frac{\partial g}{\partial x^{k}}+
        \frac{1}{2}\sum_{k,l=1}^{d_{X}}
        \sum_{r=1}^{d_{V}}\sigma_{k,r}
        \sigma_{l,r}
        \frac{\partial^{2}g}{\partial x^{k}\partial x^{l}} \;\text{ and}\\
    L^{r}g & =\sum_{k=1}^{d_{X}}
        \sigma_{k,r}
        \frac{\partial g}{\partial x^{k}},\quad r=1,...,d_{V}.
\end{align*}
Lastly, for $\alpha=(\alpha_{1},...,\alpha_{k})\in\mathcal{M}$,
the differential operator $L^{\alpha}$ is defined to be the composition
$L^{\alpha} = L^{\alpha_{1}}\circ \cdots\circ L^{\alpha_{k}}$,
where, by convention, $L^{\emptyset}g=g$.

\subsubsection{Discretisation of the log-likelihood process}
\label{sec:disc_log_likelihood}
With the stochastic Taylor expansion at hand, we can now describe the
discretisation of the log-likelihood in \eqref{eq:log-likelihood}.
To this end, let for all $t>0$,
\begin{equation*}
    \Pi(t) =\big\lbrace \lbrace t_0,\ldots,t_n  \rbrace \subset [0,t]^{n+1}:
    0=t_{0}< t_1 <\cdots<t_{n}=t,\; n = 1, 2, \ldots \big\rbrace
\end{equation*}
be the set of all partitions of the interval $\left[0,t\right]$.
For a given partition we call the quantity
$\delta=\max\{t_{j+1}-t_j : j=0,\ldots,n-1\}$ the
\emph{meshsize} of $\tau$.
Then we discretise the log-likelihood as follows.
For all $t>0$, $\tau\in\Pi(t)$ and all positive integers $m$ we consider
\begin{align*}
    \xi_{t}^{\tau,m} & =
        \sum_{j=0}^{n-1}\xi_{t}^{\tau,m}(j)=
        \sum_{j=0}^{n-1}\sum_{i=0}^{d_{Y}}
        \sum_{\alpha\in\mathcal{M}_{0,m-1}}
            L^{\alpha}\hat{h}_{i}(X_{t_{j}})
            \int_{t_{j}}^{t_{j+1}}I_{\alpha}(\mathbf{1})
            {}_{t_{j},s}d\hat{Y}_{s}^{i}\\
    &=
    \sum_{j=0}^{n-1}
    \bigl\lbrace
    \kappa_j^{0,m}
    +
    \int_{t_{j}}^{t_{j+1}}
        \big\langle \eta_j^{0,m}(s), \mathrm{d}{Y}_s \big\rangle
        \bigr\rbrace,
\end{align*}
where we define for all integers $l\leq m-1$ and $j=0,\ldots,n-1$ the quantities
\begin{align*}
    \kappa^{l,m} &=
    \sum_{j=0}^{n-1}
    \kappa_j^{l,m} =
    \sum_{j=0}^{n-1}
    \Big\lbrace
    -\frac{1}{2}
    \sum_{\alpha\in\mathcal{M}_{l,m-1}}
    L^{\alpha}\langle h(\cdot), h(\cdot) \rangle(X_{t_j})
    \int_{t_j}^{t_{j+1}} I_{\alpha} (\mathbf{1})_{t_{j},s}
    \,\mathrm{d}s
    \Big\rbrace\\
    \eta_j^{l,m}(s) &=
        \biggl(
        \sum_{\alpha\in\mathcal{M}_{l,m-1}}
        L^{\alpha}h_{i} (X_{t_{j}} )
        I_{\alpha} (\mathbf{1})_{t_{j},s}
        \biggr)_{i=1,\ldots,d_Y}.\\
\end{align*}
and $\langle \cdot, \cdot \rangle$ denotes the euclidean inner product.
Note that by setting, in the case of $m>2$,
\begin{align*}
    \mu^{\tau,m}\left(j\right)&=\sum_{i=0}^{d_{Y}}
    \sum_{\alpha\in\mathcal{M}_{2,m-1}}
    L^{\alpha}\hat{h}_{i}(X_{t_{j}})\int_{t_{j}}^{t_{j+1}}I_{\alpha}
    (\mathbf{1})_{t_{j},s}d\hat{Y}_{s}^{i} \\
    &=
    \kappa_j^{2,m}
    +
    \int_{t_{j}}^{t_{j+1}}
    \big\langle\eta_{j}^{2,m}(s), \mathrm{d}Y_s \big\rangle,
\end{align*}
we may write the above as
\begin{equation*}
    \xi_{t}^{\tau,m}=
        \xi_{t}^{\tau,2}+\sum_{j=0}^{n-1}\mu^{\tau,m}\left(j\right).
\end{equation*}

As outlined before, the discretisations $\xi^{\tau,m}$ are obtained by replacing
the processes $\big(\hat{h}_i(X_t)\big){}_{t\geq 0}$, $i = 0, \ldots, d_Y$ in
\eqref{eq:log-likelihood} with the truncation of degree
$m-1$ of the corresponding stochastic Taylor expansion
of $\hat{h}_{i}\left(X_{t}\right)$. These discretisations are subsequently used
to obtain discretisation schemes of first and second order for the filter
$\pi_{t}(\varphi)$.
However, they cannot be used directly to produce discretisation schemes
of any order $m>2$ because they do not have finite exponential moments
(required to define the discretisation schemes). More precisely, the quantities
$\mu^{\tau,m}\left(j\right)$ do not have finite exponential moments because of
the high order iterated integral involved.
For this, we need to introduce a truncation of $\mu^{\tau,m}\left(j\right)$
resulting in a (partial) taming procedure to the stochastic Taylor expansion of
$\big(\hat{h}_i(X_t)\big){}_{t\geq 0}$.
To achieve this, we introduce for every positive integer $q$ and all $\delta>0$
the truncation functions
$\Gamma_{q, \delta}\colon \mathrm{R}\to\mathrm{R}$ such that
\begin{equation}
\label{eq:trunc_fct}
    \Gamma_{q,\delta}\left(z\right)=\frac{z}{1+\left(z/\delta\right)^{2q}}
\end{equation}
and set, for all $j=0,...,n-1$,
\begin{equation*}
    \bar{\xi}_{t}^{\tau,m}\left(j\right)=
    \begin{cases}
        \displaystyle \xi_{t}^{\tau,m}\left(j\right), & \mathrm{if}\ \ m=1,2\\
        \displaystyle \xi_{t}^{\tau,2}\left(j\right)+\Gamma_{m,(t_{j+1}-t_j)}
            \left(\mu^{\tau,m}\left(j\right)\right), & \mathrm{if}\ \ m>2
    \end{cases}\;.
\end{equation*}
Utilising the above, the truncated discretisations of the log-likelihood finally
read
\begin{equation}
\label{eq:trunc_xi}
    \bar{\xi}_{t}^{\tau,m}=\sum_{j=0}^{n-1}\bar{\xi}_{t}^{\tau,m}\left(j\right).
\end{equation}
We end this section with a remark about the properties of the truncation
function before we go on to discretising the filter.

\begin{remark}
\label{rmk:truncation}
The following two properties of the truncation function $\Gamma$, defined in
\eqref{eq:trunc_fct}, are readily checked.
For all positive integers $q$ and all $\delta>0$ we have that
\begin{enumerate}[label=\roman*)]
    \item
    the truncation function is bounded, specifically, for all $z\in\mathrm{R}$,
    \begin{equation*}
        \big\lvert \Gamma_{q,\delta}(z) \big\rvert
        \leq \frac{\delta}{(2q-1)^{1/2q}},
    \end{equation*}
    \item and that its derivative is bounded for all $z\in\mathrm{R}$ as
    \begin{equation*}
        \frac{q(1-q)-1}{2q}
        \leq
        \frac{\mathrm{d}}{\mathrm{d}z}\Gamma_{q,\delta}(z)
        \leq 1.
    \end{equation*}
    In particular, the truncation function is Lipschitz continuous.
\end{enumerate}
\end{remark}

\subsubsection{Discretisation of the filter}
\label{sec:disc_filt}
Since $\bar{\xi}_{t}^{\tau,m}$ in \eqref{eq:trunc_xi} is a discretisation of the
log-likelihood we will now consider,
for all $t>0$, $\tau\in\Pi\left(t\right)$ and all positive integers $m$, the
discretised likelihood
\begin{equation*}
    Z_{t}^{\tau,m} =\exp\left(\bar{\xi}_{t}^{\tau,m}\right).
\end{equation*}
The filter is now discretised, under the condition that the Borel measurable
function $\varphi$ satisfies
$\tilde{\mathrm{E}}\big[\lvert \varphi(X_{t})Z_{t}^{\tau,m} \rvert\big]
< \infty$,
to the $m$-th order by
\begin{equation*}
    \rho_{t}^{\tau,m}\left(\varphi\right)  =
    \mathrm{\tilde{E}}
    \left[
    \varphi(X_{t})Z_{t}^{\tau,m} \bigm\vert \mathcal{Y}_{t}\right]
\end{equation*}
and
\begin{equation}
\label{eq:disc_filter}
    \pi_{t}^{\tau,m}(\varphi)=\frac
    {\rho_{t}^{\tau,m}(\varphi)}
    {\rho_{t}^{\tau,m}(\boldsymbol{1})}.
\end{equation}
It remains to show that the achieved discretisation is indeed of order $m$.

\subsubsection{Order of approximation for the filtering functionals}
\label{sec:order}
In the framework developed thus far, we can state the main result
of~\cite{crisan2019high} which justifies the construction and proves the high
order approximation. To this end, we consider the $L^p$-norms
$\lVert\cdot\rVert_{L^p} = \tilde{\mathrm{E}}[\lvert\cdot\rvert^p]^{1/p}$,
$p\geq 1$.

\begin{theorem}[Theorem 2.3 in~\cite{crisan2019high}]
\label{thm: Main Filtering_2}
\noindent
    Let $m$ be a positive integer, let $t>0$, let $\varphi$ be an $(m+1)$-times
    continuously differentiable function with at most polynomial growth and
    assume further that the coefficients of the partially observed system
    $(X,Y)$ in \eqref{eq:system} satisfy that
    \begin{itemize}[label=$\circ$]
        \item $f$ is bounded and
            $\max\{2, 2m-1\}$-times continuously differentiable with bounded
            derivatives,
        \item $\sigma$ is bounded and $2m$-times continuously differentiable
        with bounded derivatives,
        \item $h$ is bounded and $(2m+1)$-times continuously differentiable
        with bounded derivatives,
            and that
        \item $X_{0}$ has moments of all orders.
    \end{itemize}
    Then there exist positive constants $\delta_{0}$ and $C$, such that for all
    partitions $\tau\in\Pi(t)$ with meshsize $\delta < \delta_0$
    we have that
    \begin{equation*}
        \bigl\lVert
        \rho_{t}(\varphi) - \rho_{t}^{\tau,m}(\varphi)
        \bigr\rVert_{L^2}
        \leq
        C\delta^{m}.
    \end{equation*}
Moreover, there
exist positive constants $\bar{\delta_{0}}$ and $\bar{C}$, such that for all
partitions $\tau\in\Pi(t)$ with meshsize $\delta < \bar{\delta_{0}}$,
\begin{equation*}
    {\mathrm{E}}\Bigl[
    \big\lvert
    \pi_{t}(\varphi)-\pi_{t}^{\tau,m}(\varphi)
    \big\rvert
    \Bigr]
    \leq\bar{C}\delta^{m}.
\end{equation*}
\end{theorem}

\begin{remark}
Under the above assumption that $h$ is bounded and
 $\varphi$ has at most polynomial growth, the required condition
 from Theorem~2.4 in \cite{crisan2019high} that
  there exists $\varepsilon>0$ such that
$\sup_{\left\lbrace\tau\in\Pi(t):\delta<\delta_0\right\rbrace}
\big\lVert \pi_{t}^{\tau,m}(\varphi) \big\rVert_{L^{2+\varepsilon}}
<\infty$ holds.
\end{remark}

\section{Robustness of the approximation}
\label{sec:robustness}
The classical robustness of the filter as in Theorem~5.12
in~ \cite{bain2009fundamentals} states that
for every $t>0$ and bounded Borel measurable function $\varphi$ the filter
$\pi_t(\varphi)$ can be represented as a function of the observation \emph{path}
\begin{equation*}
    Y_{[0,t]}(\omega) = \lbrace Y_s(\omega) \colon s\in[0,t]\rbrace,
    \qquad \omega\in\Omega.
\end{equation*}
In particular, $Y_{[0,t]}$ is here a path-valued random variable.
The precise meaning of robustness is then that there exists a unique bounded
Borel measurable function $F^{t,\varphi}$ on the path space
$C([0,t];\mathrm{R}^{d_Y})$, that is the space of continuous
$\mathrm{R}^{d_Y}$-valued functions on $[0,t]$,
with the properties that
\begin{enumerate}[label=\roman*)]
\item $P$-almost surely,
\begin{equation*}
    \pi_t(\varphi) = F^{t,\varphi}(Y_{[0,t]})
\end{equation*}
and
\item $F^{t,\varphi}$ is continuous with respect to the supremum
    norm\footnote{For a subset $D\subseteq \mathrm{R}^l$ and a function
    $\psi\colon D \to \mathrm{R}^d$ we set
    $\lVert \psi \rVert_\infty
    = \max_{i=1,\ldots,d}\lVert \psi_i \rVert_\infty
    = \max_{i=1,\ldots,d}\sup_{x\in D}\lvert \psi_i(x) \rvert$}.
\end{enumerate}
The volume~\cite{bain2009fundamentals} contains further details on the robust
representation.
In the present paper, we establish the analogous result for the discretised
filter $\pi_t^{\tau,m}(\varphi)$ from \eqref{eq:disc_filter}.
It is formulated as follows.
\begin{theorem}
\label{thm:robust}
    Let $t>0$, $\tau=\{t_0,\ldots,t_n\}\in\Pi(t)$, let $m$ be a positive
    integer and let $\varphi$ be a bounded Borel measurable function.
    Then there exists a function
    $F^{\tau,m}_{\varphi}\colon C([0,t];\mathrm{R}^{d_Y}) \to \mathrm{R}$
    with the properties that
    \begin{enumerate}[label=\roman*)]
    \item $P$-almost surely,
        \begin{equation*}
            \pi_t^{\tau,m}(\varphi) = F^{\tau,m}_{\varphi}(Y_{[0,t]})
        \end{equation*}
        and
    \item for every two bounded paths $y_1, y_2 \in C([0,t];\mathrm{R}^{d_Y})$
        there exists a positive constant $C$ such that
        \begin{equation*}
            \big\rvert F^{\tau,m}_{\varphi}(y_1) -  F^{\tau,m}_{\varphi}(y_2)
            \big\rvert
            \leq C \lVert \varphi \rVert_\infty \lVert y_1-y_2 \rVert_\infty.
        \end{equation*}
    \end{enumerate}
\end{theorem}
Note that Theorem~\ref{thm:robust} implies the following statement in the
total variation norm.

\begin{corollary}
Let $t>0$, $\tau=\{t_0,\ldots,t_n\}\in\Pi(t)$, and let $m$ be a positive
    integer.
    Then, for every two bounded paths $y_1, y_2 \in C([0,t];\mathrm{R}^{d_Y})$
    there exists a positive constant $C$ such that
    \begin{equation*}
        \big\lVert \pi_t^{\tau,m,y_1} -
            \pi_t^{\tau,m,y_2}\big\rVert_\text{TV}
        =
        \sup_{\varphi \in B_b, \lVert \varphi \rVert_\infty \leq 1}
        \big\lvert F^{\tau,m}_{\varphi}(y_1) -
            F^{\tau,m}_{\varphi}(y_2)\big\rvert
        \leq C \lVert y_1-y_2 \rVert_\infty,
    \end{equation*}
where $B_b$ is the set of bounded and Borel measurable functions.
\end{corollary}
\begin{remark}
A natural question that arises in this context is to seek the rate of
pathwise convergence of $F_\varphi^{\tau,m}$ to $F_\varphi$ (defined as the 
limit
of $F_\varphi^{\tau,m}$ when the meshsize goes to zero) as functions on the path
space. The rate of pathwise convergence is expected to be dependent on the
H\"older constant of the observation path. Therefore, it is expected to be not
better than $\frac{1}{2}-\epsilon$ for a semimartingale observation.
The absence of high order iterated integrals of the observation process
in the construction of $F_\varphi^{\tau,m}$ means that
one cannot obtain \emph{pathwise} high order approximations based on the work in
\cite{crisan2019high}. Such approximations will no longer be continuous in the
supremum norm. Thus we need to consider rough path norms in this context.
In a different setting, Clark showed in the earlier paper~\cite{clark1980rate}
that one cannot construct pathwise
approximations of solutions of SDEs by using only increments of the driving
Brownian motion.
\end{remark}
In the following and final part of the paper, we exhibit the proof of
Theorem~\ref{thm:robust}.

\section{Proof of the robustness of the approximation}
\label{sec:proof}
We begin by constructing what will be the robust representation.
Consider, for all $y\in C([0,t];\mathrm{R}^{d_Y})$,
\begin{align*}
    \Xi_{t}^{\tau,m}(y)
        &=
        \sum_{j=0}^{n-1}
        \big\lbrace
        \kappa_j^{0,m}
        + \big\langle \eta_j^{0,m}(t_{j+1}), y_{t_{j+1}} \big\rangle
        - \big\langle \eta_j^{0,m}(t_{j}), y_{t_{j}} \big\rangle
        - \int_{t_j}^{t_{j+1}}
        \big\langle y_s , \mathrm{d}\eta_j^{0,m}(s) \big\rangle
        \big\rbrace\\
        &=
        \sum_{j=0}^{n-1}
        \big\lbrace
        \kappa_j^{0,m}
        + \big\langle \eta_j^{0,m}(t_{j+1}), y_{t_{j+1}}\big\rangle
        - \big\langle h(X_{t_{j}}), y_{t_{j}} \big\rangle
        - \int_{t_j}^{t_{j+1}} \big\langle y_s ,\mathrm{d}\eta_j^{0,m}(s) 
        \big\rangle
        \big\rbrace\\
        &=
        \big\langle h(X_{t_{n}}), y_{t_{n}} \big\rangle
        - \big\langle h(X_{t_{0}}), y_{t_{0}} \big\rangle\\
        &+
        \sum_{j=0}^{n-1}
        \big\lbrace
        \kappa_j^{0,m}
        + \big\langle \eta_j^{0,m}(t_{j+1})-h(X_{t_{j+1}}), y_{t_{j+1}}
        \big\rangle
        - \int_{t_j}^{t_{j+1}} \big\langle y_s ,\mathrm{d}\eta_j^{0,m}(s) 
        \big\rangle
        \big\rbrace
\end{align*}
and further, for $m>2$,
\begin{equation*}
    M^{\tau,m}_j(y)
    =
        \kappa_j^{2,m} +
        \big\langle \eta_j^{2,m}(t_{j+1}), y_{t_{j+1}}\big\rangle
        -
        \int_{t_j}^{t_{j+1}}
        \big\langle
        y_s, d\eta_j^{2,m}(s)
        \big\rangle
\end{equation*}
so that we can define
\begin{equation*}
\bar{\Xi}_{t}^{\tau,m}(y)=
    \begin{cases}
        \displaystyle \Xi_{t}^{\tau,m}(y), & \text{if}\ \ m=1,2\\
        \displaystyle \Xi_{t}^{\tau,2}(y)+\sum_{j=0}^{n-1}
        \Gamma_{m,(t_{j+1}-t_j)}
        \bigl(M^{\tau,m}_j(y)\bigr), & \text{if}\ \ m>2
    \end{cases}\;.
\end{equation*}
Furthermore, set
\begin{equation*}
    \mathcal{Z}^{\tau,m}_t(y) = \exp\bigl(\bar{\Xi}_{t}^{\tau,m}(y)\bigr).
\end{equation*}
\begin{example}
The robust approximation for $m=1$ and $m=2$ are given as follows.
First, if $m=1$, then
\begin{align*}
\Xi_{t}^{\tau,1}(y)
    &=
        \sum_{j=0}^{n-1}
        \big\lbrace
        \kappa_j^{0,1}
        + \big\langle \eta_j^{0,1}(t_{j+1}), y_{t_{j+1}}\big\rangle
        - \big\langle h(X_{t_{j}}), y_{t_{j}} \big\rangle
        - \int_{t_j}^{t_{j+1}} \big\langle y_s ,\mathrm{d}\eta_j^{0,1}(s) 
        \big\rangle
        \big\rbrace\\
    &= \sum_{j=0}^{n-1}
        \big\lbrace
        -\frac{1}{2}\langle h,h \rangle(X_{t_j})(t_{j+1}-t_j)
        + \big\langle h(X_{t_{j}}), y_{t_{j+1}} - y_{t_{j}}\big\rangle
        \big\rbrace
\end{align*}
and also $\bar{\Xi}_{t}^{\tau,1}(y) = \Xi_{t}^{\tau,1}(y)$ so that
$\mathcal{Z}^{\tau,1}_t(y) = \exp\bigl({\Xi}_{t}^{\tau,1}(y)\bigr)$.
If $m=2$, then
\begin{align*}
\Xi_{t}^{\tau,2}(y) &= \Xi_{t}^{\tau,1}(y) +
\sum_{j=0}^{n-1}
    \big\lbrace
    \kappa_j^{1,2} +
\big\langle \eta_j^{1,2}(t_{j+1}), y_{t_{j+1}}\big\rangle
- \int_{t_j}^{t_{j+1}} \big\langle y_s ,\mathrm{d}\eta_j^{1,2}(s) \big\rangle
        \big\rbrace \\
    &=
    \Xi_{t}^{\tau,1}(y) -
    \sum_{\alpha\in \mathcal{M}_1}
    \sum_{j=0}^{n-1}
    \frac{1}{2}L^{\alpha}\langle h, h \rangle(X_{t_j})
    \int_{t_j}^{t_{j+1}} V_s^{\alpha}-V_{t_j}^\alpha\,\mathrm{d}s
     \\
    &+
    \sum_{\alpha\in \mathcal{M}_1}
    \sum_{j=0}^{n-1}
    \int_{t_j}^{t_{j+1}}
    \bigl\langle L^\alpha h(X_{t_j}), y_{t_{j+1}} - y_s \bigr\rangle
    \,\mathrm{d}V_s^\alpha.
\end{align*}
Therefore, also $\bar{\Xi}_{t}^{\tau,2}(y) = \Xi_{t}^{\tau,2}(y)$ so that
$\mathcal{Z}^{\tau,2}_t(y) = \exp\bigl({\Xi}_{t}^{\tau,2}(y)\bigr)$.
\end{example}

First, we show that the newly constructed $\mathcal{Z}^{\tau,m}_t$ is locally
bounded.
\begin{lemma}
\label{lem:Z_bound}
    Let $t>0$, let $\tau=\{t_0,\ldots,t_n\}\in\Pi(t)$ be a partition with mesh
    size $\delta$ and let $m$ be a positive integer.
    Then, for all $R>0$, $p\geq 1$ there exists a positive constant $B_{p,R}$
    such that
    \begin{equation*}
        \sup_{\lVert y \rVert_\infty \leq R}
        \big\lVert \mathcal{Z}^{\tau,m}_t(y) \big\rVert_{L^p} \leq B_{p,R}.
    \end{equation*}
\end{lemma}

\begin{proof}
    Notice that, by Remark \ref{rmk:truncation}, in the case $m\geq 2$,
    we have for all
    $y\in C([0,t];\mathrm{R}^{d_Y})$ that
    \begin{equation*}
        \bar{\Xi}_{t}^{\tau,m}(y) \leq
        \Xi_{t}^{\tau,2}(y)+\frac{n\delta}{(2m-1)^{1/2m}}.
    \end{equation*}
    This implies that for all
    $y\in C([0,t];\mathrm{R}^{d_Y})$,
    \begin{equation*}
        \mathcal{Z}^{\tau,m}_t(y) =
            \exp\bigl(\bar{\Xi}_{t}^{\tau,m}(y)\bigr) \leq
            \exp\bigl(\Xi_{t}^{\tau,2}(y)\bigr)
            \exp\biggl(\frac{n\delta}{(2m-1)^{1/2m}}\biggr).
    \end{equation*}
    For $m=1$, we clearly have
    $
        \mathcal{Z}^{\tau,1}_t(y) = \exp\bigl(\Xi_{t}^{\tau,1}(y)\bigr).
    $
Hence, it suffices to show the result for $m=1,2$ only.
We have
\begin{align*}
    \Xi_{t}^{\tau,2}(y)
    & =
    \Xi_{t}^{\tau,1}(y)
    +
    \sum_{j=0}^{n-1}
    \big\lbrace
    \kappa_j^{1,2}
    +
    \big\langle \eta_j^{1,2}(t_{j+1}), y_{t_{j+1}}\big\rangle
    -\int_{t_j}^{t_{j+1}} \big\langle y_s ,\mathrm{d}\eta_j^{1,2}(s) \big\rangle
    \big\rbrace.
\end{align*}
Now, by the triangle inequality, boundedness of $y$,
and boundedness of $h$, we get
\begin{align*}
    \big\lvert \Xi_t^{\tau,1}(y) \big\rvert &=
    \bigg\lvert
    \sum_{j=0}^{n-1}
        \big\lbrace
        \kappa_j^{0,1}
        + \big\langle \eta_j^{0,1}(t_{j+1}), y_{t_{j+1}}\big\rangle
        - \big\langle h(X_{t_{j}}), y_{t_{j}} \big\rangle
        - \int_{t_j}^{t_{j+1}} \big\langle y_s ,\mathrm{d}\eta_j^{0,1}(s)
        \big\rangle
        \big\rbrace
        \bigg\rvert \displaybreak[1]\\
    &=
        \bigg\lvert
    \sum_{j=0}^{n-1}
        \big\lbrace
        \kappa_j^{0,1}
        + \big\langle h(X_{t_{j}}), y_{t_{j+1}}-y_{t_{j}}\big\rangle
        \big\rbrace
        \bigg\rvert \displaybreak[1]\\
    &=
    \bigg\lvert
    \sum_{j=0}^{n-1}
    \big\lbrace
    -\frac{1}{2}\langle h(X_{t_j}), h(X_{t_j})\rangle
    (t_{j+1}-t_j)
    +\langle h(X_{t_{j}}), y_{t_{j+1}}-y_{t_{j}} \rangle
    \big\rbrace
    \bigg\rvert \displaybreak[1]\\
    &\leq
   \frac{t d_Y \lVert h \rVert_\infty^2 }{2}
    + 2 R \lVert h \rVert_\infty = C_0,
\end{align*}
where we denote the final constant by $C_0$.
Furthermore, by the triangle inequality, boundedness of $y$,
and boundedness of $h$ and its derivatives,
\begin{align*}
    & \bigg\lvert
    \sum_{j=0}^{n-1}
    \big\lbrace
    \kappa_j^{1,2}
    +\big\langle \eta_j^{1,2}(t_{j+1}), y_{t_{j+1}}\big\rangle
    -\int_{t_j}^{t_{j+1}} \big\langle y_s ,\mathrm{d}\eta_j^{1,2}(s)\big\rangle
    \big\rbrace
    \bigg\rvert \displaybreak[0]\\
    &=
    \sum_{\alpha\in \mathcal{M}_1}
    \Biggl\lbrace
    \bigg\lvert
    \sum_{j=0}^{n-1}
    \frac{1}{2}L^{\alpha}\langle h, h \rangle(X_{t_j})
    \int_{t_j}^{t_{j+1}} V_s^{\alpha}-V_{t_j}^\alpha\,\mathrm{d}s
     \\
    &+
    \sum_{j=0}^{n-1}
    \int_{t_j}^{t_{j+1}}
    \bigl\langle L^\alpha h(X_{t_j}), y_{t_{j+1}} - y_s \bigr\rangle
    \,\mathrm{d}V_s^\alpha
    \bigg\rvert
    \Biggr\rbrace \displaybreak[1]\\
    &\leq
    \sum_{\alpha\in \mathcal{M}_1\setminus\{0\}}
    \Biggl\lbrace
    \bigg\lvert
    \sum_{j=0}^{n-1}
    \int_{t_j}^{t_{j+1}}
    \frac{1}{2}L^{\alpha}\langle h, h\rangle(X_{t_j})
    (t_{j+1}-s)
    +
    \bigl\langle L^\alpha h(X_{t_j}), y_{t_{j+1}} - y_s \bigr\rangle
    \,\mathrm{d}V_s^{\alpha}
    \bigg\rvert
    \Biggr\rbrace\\
    &+
    \bigg\lvert
    \sum_{j=0}^{n-1}
    \int_{t_j}^{t_{j+1}}
    \frac{1}{2}L^{0}\langle h, h\rangle(X_{t_j})
     (s-t_j) + \bigl\langle L^0 h(X_{t_j}), y_{t_{j+1}} - y_s \bigr\rangle
     \,\mathrm{d}s
    \bigg\rvert
     \displaybreak[0]\\
    &\leq
    \sum_{\alpha\in \mathcal{M}_1\setminus\{0\}}
    \Biggl\lbrace
    \bigg\lvert
    \sum_{j=0}^{n-1}
    \int_{t_j}^{t_{j+1}}
    \frac{1}{2}L^{\alpha}\langle h, h\rangle(X_{t_j})
    (t_{j+1}-s)
    +
    \bigl\langle L^\alpha h(X_{t_j}), y_{t_{j+1}} - y_s \bigr\rangle
    \,\mathrm{d}V_s^{\alpha}
    \bigg\rvert
    \Biggr\rbrace\\
    &+
    \frac{1}{2}\delta t
    \lVert L^{0}\langle h, h\rangle \rVert_\infty
    +
    2 d_Y R t \lVert L^{0} h\rVert_\infty
    \displaybreak[1]\\
    &=
    C_1 +
    \sum_{\alpha\in \mathcal{M}_1\setminus\{0\}}
    \Biggl\lbrace
    \bigg\lvert
    \int_{0}^{t}
    \frac{1}{2}L^{\alpha}\langle h, h\rangle(X_{\lfloor s \rfloor})
    (\lceil s \rceil -s)
    +
    \bigl\langle L^\alpha h(X_{\lfloor s \rfloor}), y_{\lceil s \rceil} -
    y_s \bigr\rangle
    \,\mathrm{d}V_s^{\alpha}
    \bigg\rvert
    \Biggr\rbrace.
\end{align*}
Here, $C_1$ is a constant introduced for conciseness.
Then,
\begin{align*}
    &\bigl\lVert
    \mathcal{Z}^{\tau,2}_t(y)
    \bigr\rVert_{L^p} \displaybreak[0]\\
    &=
    \biggl\lVert
    \mathcal{Z}^{\tau,1}_t(y)
    \exp\bigl(
    \sum_{j=0}^{n-1}
    \big\lbrace
    \kappa_j^{1,2}
    -\big\langle \eta_j^{1,2}(t_{j+1}), y_{t_{j+1}}\big\rangle
    -\int_{t_j}^{t_{j+1}} \big\langle y_s ,\mathrm{d}\eta_j^{1,2}(s)\big\rangle
    \big\rbrace
    \bigr)
    \biggr\rVert_{L^p}\displaybreak[1]\\
    &\leq
    \exp\bigl(C_0+C_1\bigr)\\
    &\Biggl\lVert
    \exp\Bigl(
    \sum_{\alpha\in \mathcal{M}_1\setminus\{0\}}
    \Biggl\lbrace
    \bigg\lvert
    \int_{0}^{t}
    \frac{1}{2}L^{\alpha}\langle h, h\rangle(X_{\lfloor s \rfloor})
    (\lceil s \rceil -s)
    +
    \bigl\langle L^\alpha h(X_{\lfloor s \rfloor}), y_{\lceil s \rceil} - y_s 
    \bigr\rangle
    \,\mathrm{d}V_s^{\alpha}
    \bigg\rvert
    \Biggr\rbrace
    \Bigr)
    \Biggr\rVert_{L^p}\displaybreak[1]\\
    &<
    \infty.
\end{align*}
The lemma is thus proved.
\end{proof}

In analogy to the filter, we define the functions
\begin{equation*}
    G^{\tau,m}_{\varphi}(y) =
    \tilde{\mathrm{E}}[\varphi(X_t)\mathcal{Z}^{\tau,m}_t(y)]
\end{equation*}
and
\begin{equation*}
    F^{\tau,m}_{\varphi}(y) =
    \frac{G^{\tau,m}_{\varphi}(y)}{G^{\tau,m}_{\mathbf{1}}(y)} =
    \frac{\tilde{\mathrm{E}}[\varphi(X_t)\mathcal{Z}^{\tau,m}_t(y)]}
    {\tilde{\mathrm{E}}[\mathcal{Z}^{\tau,m}_t(y)]}.
\end{equation*}

\begin{lemma}
\label{lem:G_lipschitz}
    Let $\tau\in\Pi(t)$ be a partition,
    let $m$ be a positive integer and let
    $\varphi$ be a bounded Borel measurable function.
    Then the functions
    $G_\varphi^{\tau,m}\colon C([0,t];\mathrm{R}^{d_Y})\to \mathrm{R}$ and
    $F_\varphi^{\tau,m}\colon C([0,t];\mathrm{R}^{d_Y})\to \mathrm{R}$
    are locally Lipschitz continuous and locally bounded.
    Specifically, for every two paths
    $y_1, y_2 \in C([0,t];\mathrm{R}^{d_Y})$ such that there
    exists a real number $R>0$ with $\lVert y_1 \rVert_\infty \leq R$ and
    $\lVert y_2 \rVert_\infty \leq R$, there exist constants
    $L_G$, $M_G$, $L_F$, and $M_F$ such that
    \begin{equation*}
        \big\lvert
        G_\varphi^{\tau,m}(y_1) - G_\varphi^{\tau,m}(y_2)
        \big\rvert
        \leq
        L_G
        \lVert \varphi \rVert_\infty
        \lVert y_1-y_2 \rVert_\infty
        \qquad
        \text{and}
        \qquad
        \big\lvert
        G_\varphi^{\tau,m}(y_1)
        \big\rvert
        \leq
        M_G \lVert \varphi \rVert_\infty
    \end{equation*}
    and
    \begin{equation*}
        \big\lvert
        F_\varphi^{\tau,m}(y_1) - F_\varphi^{\tau,m}(y_2)
        \big\rvert
        \leq
        L_F
        \lVert \varphi \rVert_\infty
        \lVert y_1-y_2 \rVert_\infty
        \qquad
        \text{and}
        \qquad
        \big\lvert
        F_\varphi^{\tau,m}(y_1)
        \big\rvert
        \leq
        M_F \lVert \varphi \rVert_\infty.
    \end{equation*}
\end{lemma}
\begin{proof}
We first show the results for $G_\varphi^{\tau,m}$.
Note that
\begin{equation*}
    \big\lvert
    \mathcal{Z}_t^{\tau,m}(y_1) -
    \mathcal{Z}_t^{\tau,m}(y_2)
    \big\rvert
    \leq
    \big(
    \mathcal{Z}_t^{\tau,m}(y_1) + \mathcal{Z}_t^{\tau,m}(y_2)
    \big)
    \bigl\lvert
    \bar{\Xi}_t^{\tau,m}(y_1) - \bar{\Xi}_t^{\tau,m}(y_2)
    \bigr\rvert.
\end{equation*}
Then, by the Cauchy-Schwarz inequality, for all $p\geq 1$ we have
\begin{equation}
    \big\lVert
    \varphi(X_t)
    \mathcal{Z}_t^{\tau,m}(y_1) -
    \varphi(X_t) \mathcal{Z}_t^{\tau,m}(y_2)
    \big\rVert_{L^p}
    \leq
    2B_{2p,R} \lVert \varphi \rVert_\infty
    \bigl\lVert
    \bar{\Xi}_t^{\tau,m}(y_1) - \bar{\Xi}_t^{\tau,m}(y_2)
    \bigr\rVert_{L^{2p}}.
\end{equation}
Thus, for $m>2$, we can exploit the effect of the truncation function and,
similarly to the proof of Lemma~\ref{lem:Z_bound}, it suffices to show the
result for $m=1,2$.
To this end, consider for all $q\geq 1$,
\begin{multline*}
    \bigl\lVert
    {\Xi}_t^{\tau,2}(y_1) - {\Xi}_t^{\tau,2}(y_2)
    \bigr\rVert_{L^{q}}
    \leq
    \bigl\lVert
    {\Xi}_t^{\tau,1}(y_1) - {\Xi}_t^{\tau,1}(y_2)
    \bigr\rVert_{L^{q}}\\
    +
    \Biggl\lVert
    \sum_{j=0}^{n-1}
    \big\lbrace
    \big\langle \eta_j^{1,2}(t_{j+1}), y_1({t_{j+1}})-y_2({t_{j+1}}) \big\rangle
    -\int_{t_j}^{t_{j+1}} \big\langle y_1(s)-y_2(s) ,\mathrm{d}\eta_j^{1,2}(s)
    \big\rangle
    \big\rbrace
    \Biggr\rVert_{L^{q}}.
\end{multline*}
First, we obtain for all $q\geq 1$,
\begin{multline*}
    \bigl\lVert
    {\Xi}_t^{\tau,1}(y_1) - {\Xi}_t^{\tau,1}(y_2)
    \bigr\rVert_{L^{q}}
    %%%
    =
    \Biggl\lVert
    \sum_{j=0}^{n-1}
        \langle
    h(X_{t_{j}}), (y_1(t_{j+1})-y_2(t_{j+1}))- (y_1(t_{j})-y_2(t_{j}))
        \rangle
    \Biggr\rVert_{L^{q}}\\
    %%%
    \leq
    2 d_Y \lVert h \rVert_\infty
    \lVert y_1-y_2 \rVert_\infty.
\end{multline*}
And second we have for all $q\geq 1$ that
\begin{align*}
    &\Biggl\lVert
    \sum_{j=0}^{n-1}
    \big\lbrace
    \int_{t_j}^{t_{j+1}} \big\langle
        (y_1({t_{j+1}})-y_1(s))-(y_2({t_{j+1}})-y_2(s)),
        \mathrm{d}\eta_j^{1,2}(s)
        \big\rangle
    \big\rbrace
    \Biggr\rVert_{L^{q}}\\
    &\leq
    \sum_{j=0}^{n-1}
    \Biggl\lVert
    \bigg\lvert
    \big\langle L^0 h(X_{t_j}), y_1(t_{j+1})-y_2(t_{j+1})\big\rangle
    ( {t_{j+1}} - {t_j})
    \bigg\rvert \\
    &+
    \bigg\lvert
    \int_{t_j}^{t_{j+1}}
    \bigl\langle L^0 h(X_{t_j}), y_1(s)-y_2(s) \bigr\rangle
    \,\mathrm{d}s
    \bigg\rvert \\
    &+
    \sum_{\alpha\in \mathcal{M}_1\setminus\{0\}}
    \bigg\lvert
    \big\langle L^\alpha h(X_{t_j}), y_1(t_{j+1})-y_2(t_{j+1})\big\rangle
    \bigl( V_{t_{j+1}}^\alpha - V_{t_j}^\alpha\bigr)
    \bigg\rvert \\
    &+
    \bigg\lvert
    \int_{t_j}^{t_{j+1}}
    \bigl\langle L^\alpha h(X_{t_j}), y_1(s)-y_2(s) \bigr\rangle
    \,\mathrm{d}V_s^\alpha
    \bigg\rvert
    \Biggr\rVert_{L^{q}}\\
    &\leq
    \biggl[
    \bar{C}_1 + \bar{C}_2
    \sum_{j=0}^{n-1}\sum_{\alpha\in \mathcal{M}_1\setminus\{0\}}
    \lVert V^\alpha_{t_{j+1}}-V^\alpha_{t_{j}} \rVert_{L^{q}}
    \biggr]
    \lVert y_1-y_2\rVert_\infty\\
    &\leq
    C \lVert y_1-y_2\rVert_\infty
\end{align*}
This and Lemma~\ref{lem:Z_bound} imply that $G_\varphi^{\tau,m}$
is locally Lipschitz and locally bounded.
To show the result for $F_\varphi^{\tau,m}$ we need to establish that
${1}/{G_{\mathbf{1}}^{\tau,m}}$ is locally bounded.
We have, using Jensen's inequality, that for $m\geq 2$
\begin{equation*}
    {G_{\mathbf{1}}^{\tau,m}} =
    \tilde{\mathrm{E}}
    \bigl[
    \mathcal{Z}_t^{\tau,m}
    \bigr]
    \geq
    \exp
    \Bigl(
    \tilde{\mathrm{E}}\bigl[\bar{\Xi}_t^{\tau,m}\bigr]
    \Bigr)
    \geq
    \exp
    \Bigl(\tilde{\mathrm{E}}\bigl[{\Xi}_t^{\tau,2}\bigr]\Bigr)
    \exp
    \Biggl(-\frac{n\delta}{(2m-1)^{1/2m}}\Biggr)
\end{equation*}
and for $m=1$ clearly
\begin{equation*}
    {G_{\mathbf{1}}^{\tau,1}} =
    \tilde{\mathrm{E}}
    \bigl[
    \mathcal{Z}_t^{\tau,1}
    \bigr]
    \geq
    \exp
    \Bigl(
    \tilde{\mathrm{E}}\bigl[{\Xi}_t^{\tau,1}\bigr]
    \Bigr).
\end{equation*}
Since the quantities
$\tilde{\mathrm{E}}\bigl[{\Xi}_t^{\tau,1}\bigr]$
and
$\tilde{\mathrm{E}}\bigl[{\Xi}_t^{\tau,2}\bigr]$
are finite, the lemma is proved.
\end{proof}

In the following, given $t > 0$, we set for every $\gamma \in (0,1/2)$,
\begin{equation*}
    \mathcal{H}_{\gamma} = \left\{ \,
        y\in C([0,t];\mathrm{R}^{d_Y}):
        \sup_{s_1,s_2\in [0,t]}
        \frac{\|y_{s_1}-y_{s_2}\|_\infty}{|s_1-s_2|^{{\gamma}}}
        < \infty
        \,\right\}
        \subseteq C([0,t];\mathrm{R}^{d_Y})
\end{equation*}
and recall that $Y_{[0,t]}\colon\Omega \to C([0,t];\mathrm{R}^{d_Y})$
denotes the random
variable in path space corresponding to the observation process $Y$.

\begin{lemma}
    \label{lem:H13}
    For all $t > 0$ and $\gamma \in (0,1/2)$, we have $\tilde{P}$-almost surely
    that $Y_{[0,t]}\in \mathcal{H}_{\gamma}$.
\end{lemma}

\begin{proof}
Recall that, under $\tilde{P}$, the observation process $Y$ is a Brownian motion
and, by the Brownian scaling property, it suffices to show the result for
$t=1$.
Therefore, let $\gamma \in (0,1/2)$ and note that for all ${\delta}\in (0,1]$ we
have
\begin{equation*}
    \sup_{s_1,s_2\in [0,1]}
    \frac{\|Y_{s_1}-Y_{s_2}\|_\infty}{|s_1-s_2|^{{\gamma}}}
    =
    \max\left\{
    \sup_{\substack{s_1,s_2\in [0,1]\\ |s_1-s_2|\leq {\delta}}}
        \frac{\|Y_{s_1}-Y_{s_2}\|_\infty}
            {|s_1-s_2|^{{\gamma}}},\;
    \sup_{\substack{s_1,s_2\in [0,1]\\ |s_1-s_2|\geq {\delta}}}
        \frac{\|Y_{s_1}-Y_{s_2}\|_\infty}
            {|s_1-s_2|^{{\gamma}}}
    \right\}.
\end{equation*}
The second element of the maximum above is easily bounded, $\tilde{P}$-almost
surely, by the sample path continuity.
For the first element, note that there exists $\delta_0\in (0,1)$ such that for
all $\delta\in (0,\delta_0]$,
\begin{equation*}
    \delta^{\gamma} \geq \sqrt{2\delta\log(1/\delta)}.
\end{equation*}
Therefore, it follows that $\tilde{P}$-almost surely,
\begin{equation*}
    \sup_{\substack{s_1,s_2\in [0,1]\\ |s_1-s_2|\leq\delta_0}}
            \frac{\|Y_{s_1}-Y_{s_2}\|_\infty}
                {|s_1-s_2|^{\gamma}}
    \leq
    \sup_{\substack{s_1,s_2\in [0,1]\\ |s_1-s_2|\leq\delta_0}}
        \frac{\|Y_{s_1}-Y_{s_2}\|_\infty}
            {\sqrt{2|s_1-s_2|\log(1/|s_1-s_2|)}}.
\end{equation*}
The L\'evy modulus of continuity of Brownian motion
further ensures that $\tilde{P}$-almost surely,
\begin{equation*}
        \limsup_{\delta \downarrow 0}
            \sup_{\substack{s_1,s_2\in [0,1]\\ |s_1-s_2|\leq\delta}}
            \frac{\|Y_{s_1}-Y_{s_2}\|_\infty}{\sqrt{2\delta\log(1/\delta)}}
        = 1.
\end{equation*}
The Lemma~\ref{lem:H13} thus follows.
\end{proof}

\begin{lemma}
\label{lem:int_version}
    Let $\tau=\{0=t_1<\ldots<t_n=t\}\in\Pi(t)$ be a partition,
    let $j\in\{0,\ldots,n-1\}$ and let $c$ be a positive integer.
    Then, there exists a version of the stochastic integral
    \begin{equation*}
    C([0,t];\mathrm{R}^{d_Y}) \times \Omega
    \ni \;
    (y, \omega)
    \mapsto
    \int_{t_j}^{t_{j+1}} \langle y_s,\mathrm{d}\eta_j^{c,c+1}(s,\omega)\rangle
    \; \in \mathrm{R}
    \end{equation*}
    such that
    it is equal on $\mathcal{H}_\gamma\times \Omega$ to a
    $\mathcal{B}(C([0,t];\mathrm{R}^{d_Y}))\times\mathcal{F}$-measurable
    mapping.
\end{lemma}

\begin{proof}
For $k$ a positive integer,
define for $y\in C([0,t];\mathrm{R}^{d_Y})$,
\begin{equation*}
    \mathcal{J}_j^{c,k}(y) =
    \sum_{i=0}^{k-1}
    \Bigl\langle
    y_{s_{i,j}},
    \Bigl(
    \eta_j^{c,c+1}(s_{i+1,j})
    -
    \eta_j^{c,c+1}(s_{i,j})
    \Bigr)
    \Bigr\rangle,
\end{equation*}
where
$s_{i,j}= \frac{i(t_{j+1}-t_j)}{k}+t_j$,
$i=0,\ldots, k$.
Furthermore, we set $\lfloor s \rfloor = s_{i,j}$ for
$s\in [\frac{i(t_{j+1}-t_j)}{k}+t_j, \frac{(i+1)(t_{j+1}-t_j)}{k}+t_j)$.
Then, for $y \in \mathcal{H}_\gamma$, we have
\begin{align*}
    &\tilde{\mathrm{E}}
    \biggl[
    \biggl(
    \mathcal{J}_j^{c,2^l}(y) -
    \int_{t_j}^{t_{j+1}} \langle y_s,\mathrm{d}\eta_j^{c,c+1}(s)\rangle
    \biggr)^2
    \biggr]\displaybreak[1]\\
    %%%
    &=
    \tilde{\mathrm{E}}
    \biggl[
    \biggl(
    \int_{t_j}^{t_{j+1}} \langle y_{\lfloor s \rfloor} -
    y_s,\mathrm{d}\eta_j^{c,c+1}(s)\rangle
    \biggr)^2
    \biggr]\displaybreak[1]\\
    %%%
    &=
    \tilde{\mathrm{E}}
    \biggl[
    \biggl(
    \sum_{\alpha\in \mathcal{M}_c}
    \sum_{i=0}^{d_Y}
    \int_{t_j}^{t_{j+1}}
    (y_{\lfloor s \rfloor}^i - y_s^i)L^\alpha h_i(X_{t_j})
    \,\mathrm{d}I_{\alpha}(\mathbf{1})_{t_j,s}
    \biggr)^2
    \biggr]\displaybreak[1]\\
    %%%
    &\leq
    (d_V+1) d_Y \sum_{i=0}^{d_Y}
    \sum_{\alpha\in \mathcal{M}_c}
    \tilde{\mathrm{E}}
    \biggl[
    \biggl(
    \int_{t_j}^{t_{j+1}}
    (y_{\lfloor s \rfloor}^i - y_s^i)L^\alpha h_i(X_{t_j})
    \,\mathrm{d}I_{\alpha}(\mathbf{1})_{t_j,s}
    \biggr)^2
    \biggr]\displaybreak[1]\\
    %%%
    &=
    (d_V+1)
    d_Y \sum_{i=0}^{d_Y}
    \sum_{\substack{\alpha\in \mathcal{M}_c\\ \alpha_{|\alpha|}\neq 0}}
    \tilde{\mathrm{E}}
    \biggl[
    \int_{t_j}^{t_{j+1}}
    \bigl(
    (y_{\lfloor s \rfloor}^i - y_s^i) L^\alpha h_i(X_{t_j})
    \bigr)^2
    \,\mathrm{d}\langle I_{\alpha}(\mathbf{1})_{t_j,\cdot}\rangle_s
    \biggr]\\
    &+
    (d_V+1)
    d_Y \sum_{i=0}^{d_Y}
    \sum_{\substack{\alpha\in \mathcal{M}_c\\ \alpha_{|\alpha|}= 0}}
    \tilde{\mathrm{E}}
    \biggl[
    \biggl(
    \int_{t_j}^{t_{j+1}}
    (y_{\lfloor s \rfloor}^i - y_s^i) L^\alpha h_i(X_{t_j})
    \,\mathrm{d}
    \bigl[
    \int_{t_j}^{s} I_{\alpha-}(\mathbf{1})_{t_j,r} \mathrm{d}r
    \bigr]
    \biggr)^2
    \biggr]\displaybreak[1]\\
    %%%
    &\leq
    (d_V+1)
    d_Y
    \frac{K(t_{j+1}-t_j)^{2\gamma}}{2^{2l\gamma}}
    \max_{\alpha\in\mathcal{M}_c}\lVert L^\alpha h(X_{t_j}) \rVert_\infty
    \Bigg\{
    \sum_{i=0}^{d_Y}
    \sum_{\substack{\alpha\in \mathcal{M}_c\\ \alpha_{|\alpha|}\neq 0}}
    \tilde{\mathrm{E}}
    \biggl[
    \int_{t_j}^{t_{j+1}}
    (I_{\alpha-}(\mathbf{1})_{t_j,s})^2
    \,\mathrm{d}s
    \biggr]\\
    &+
    \sum_{i=0}^{d_Y}
    \sum_{\substack{\alpha\in \mathcal{M}_c\\ \alpha_{|\alpha|}= 0}}
    \tilde{\mathrm{E}}
    \biggl[
    \biggl(
    \int_{t_j}^{t_{j+1}}
    I_{\alpha-}(\mathbf{1})_{t_j,s}
    \,\mathrm{d}s
    \biggr)^2
    \biggr]
    \Bigg\}\\
    &\leq
    \frac{(d_V+1) d_Y C K(t_{j+1}-t_j)^{2\gamma}}{2^{2l\gamma}},
\end{align*}
Where the constant $C$ is independent of $l$. Thus, by Chebyshev's inequality,
we get for all $\epsilon > 0$ that
\begin{equation*}
    \tilde{P}\biggl(
    \biggl\lvert
    \mathcal{J}_j^{c,2^l}(y) -
    \int_{t_j}^{t_{j+1}} \langle y_s,\mathrm{d}\eta_j^{c,c+1}(s)\rangle
    \biggr\rvert
    >
    \epsilon
    \biggr)
    \leq
    \frac{1}{\epsilon^2}
    \frac{(d_V+1) d_Y C K(t_{j+1}-t_j)^{2\gamma}}{2^{2l\gamma}}.
\end{equation*}
However, the bound on the right-hand side is summable over $l$ so that we
conclude using the first Borel-Cantelli Lemma that, for all $\epsilon > 0$,
\begin{equation*}
    \tilde{P}\biggl(
    \limsup_{l\rightarrow \infty}
    \biggl\lvert
    \mathcal{J}_j^{c,2^l}(y) -
    \int_{t_j}^{t_{j+1}} \langle y_s,\mathrm{d}\eta_j^{c,c+1}(s)\rangle
    \biggr\rvert
    >
    \epsilon
    \biggr)
    =
    0.
\end{equation*}
Thus, for all $y \in \mathcal{H}_\gamma$, the integral $\mathcal{J}_j^{c,k}(y)$
converges $\tilde{P}$-almost surely to
$\int_{t_j}^{t_{j+1}} \langle y_s,\mathrm{d}\eta_j^{c,c+1}(s)\rangle$.
Hence, we can define the limit on $\mathcal{H}_\gamma \times \Omega$ to be
\begin{equation*}
    \mathcal{J}_j^{c}(y)(\omega)
    =
    \limsup_{l \rightarrow \infty} \mathcal{J}_j^{c,l}(y)(\omega);
    \qquad
    (y,\omega) \in \mathcal{H}_\gamma \times \Omega.
\end{equation*}
Since the mapping
\begin{equation*}
    C([0,T];\mathrm{R}^{d_Y})\times\Omega
    \ni(y,\omega)
    \mapsto
    \limsup_{l \rightarrow \infty} \mathcal{J}_j^{c,l}(y)(\omega)
    \in \mathrm{R}
\end{equation*}
is jointly $\mathcal{B}(C([0,T];\mathrm{R}^{d_Y}))\otimes\mathcal{F}$ measurable
the lemma is proved.
\end{proof}

It turns out that proving the robustness result is simplified by first 
decoupling
the processes $X$ and $Y$ in the following manner.
Let $(\mathring{\Omega},\mathring{\mathcal{F}}, \mathring{P})$ be an indentical
copy of the probability space $(\Omega,\mathcal{F}, \tilde{P})$.
Then
\begin{equation*}
    \mathring{G}^{\tau,m}_{\varphi}(y) =
    \mathring{\mathrm{E}}
    [\varphi(\mathring{X}_t)\mathring{\mathcal{Z}}^{\tau,m}_t(y)]
\end{equation*}
is the corresponding representation of $G^{\tau,m}_{\varphi}(y)$ in the new
space, where
$\mathring{\mathcal{Z}}^{\tau,m}_t(y)
=
\exp(\mathring{\bar{\Xi}}_t^{\tau,m}(y))$ with
\begin{align*}
    \mathring{\Xi}_{t}^{\tau,m}(y)
        &=
        \sum_{j=0}^{n-1}
        \mathring{\kappa}_j^{0,m}
        +  \big\langle \mathring{\eta}_j^{0,m}(t_{j+1}), y_{t_{j+1}}\big\rangle
        -  \big\langle h(\mathring{X}_{t_j}), y_{t_j}\big\rangle
        - \int_{t_j}^{t_{j+1}}
        \big\langle y_s, \,\mathrm{d}\mathring{\eta}_j^{0,m}(s) \big\rangle
\end{align*}
and, for $m>2$,
\begin{align*}
    \mathring{M}^{\tau,m}_j(y)
    &=
        \mathring{\kappa}_j^{2,m}
        -
        \big\langle \mathring{\eta}_j^{2,m}(t_{j+1}), y_{t_{j+1}}\big\rangle
        -
        \int_{t_j}^{t_{j+1}}
        \big\langle
        y_s, d\mathring{\eta}_j^{2,m}(s)
        \big\rangle,
\end{align*}
so that, finally,
\begin{equation*}
\mathring{\bar{\Xi}}_t^{\tau,m}(y)
=
\begin{cases}
        \displaystyle \mathring{\Xi}_{t}^{\tau,m}(y), & \text{if}\ \ m=1,2\\
        \displaystyle \mathring{\Xi}_{t}^{\tau,2}(y)+\sum_{j=0}^{n-1}
        \Gamma_{m,(t_{j+1}-t_j)}
        \bigl(\mathring{M}^{\tau,m}_j(y)\bigr), & \text{if}\ \ m>2.
    \end{cases}\;
\end{equation*}
Moreover, with $\mathring{\mathcal{J}}_j^{c}(y)$ corresponding to
Lemma~\ref{lem:int_version} we can write for $y\in\mathcal{H}_\gamma$,
\begin{align*}
    \mathring{\Xi}_{t}^{\tau,m}(y)
        &=
        \sum_{j=0}^{n-1}
        \mathring{\kappa}_j^{0,m}
        +  \big\langle \mathring{\eta}_j^{0,m}(t_{j+1}), y_{t_{j+1}} \big\rangle
        - \big\langle h(\mathring{X}_{t_j}), y_{t_j}\big\rangle \\
        &-\sum_{c=0}^{m-1}\sum_{j=0}^{n-1} \mathring{\mathcal{J}}_j^{c}(y).
\end{align*}
In the same way we get, \emph{mutatis mutandis}, the expression for
$\mathring{\bar{\Xi}}_t^{\tau,m}(y)$ on $\mathcal{H}_\gamma$.
Now, we denote by
\begin{equation*}
(\check{\Omega}, \check{\mathcal{F}}, \check{P}) =
    (\Omega\times\mathring{\Omega},
    \mathcal{F}\otimes\mathring{\mathcal{F}},
    \tilde{P} \otimes \mathring{P})
\end{equation*}
the product probability space. In the following we lift the processes
$\mathring{\eta}$ and
$Y$ from the component spaces to the product space by writing
$Y(\omega,\mathring{\omega}) = Y(\omega)$ and
$\mathring{\eta}_j^{c,c+1}(\omega,\mathring{\omega}) =
\mathring{\eta}_j^{c,c+1}(\mathring{\omega})$ for all
$(\omega,\mathring{\omega})\in\check{\Omega}$.

\begin{lemma}
\label{lem:null_set}
Let $c$ be a positive integer and let $j\in \{0,\ldots,n\}$.
Then there exists a nullset $N_0 \in \mathcal{F}$ such that the mapping
$ (\omega,\mathring{\omega}) \mapsto
\mathring{\mathcal{J}}_j^c(Y_{[0,t]}(\omega))(\mathring{\omega})$
coincides on $(\Omega\setminus N_0)\times \mathring{\Omega}$
with an $\check{\mathcal{F}}$-measurable map.
\end{lemma}

\begin{proof}
Notice first that
the set
\begin{equation*}
N_0 = \lbrace \omega \in \Omega \colon Y_{[0,t]}(\omega)
\notin \mathcal{H}_\gamma \rbrace
\end{equation*}
is clearly a member of $\mathcal{F}$ and we have that $\tilde{P}(N_0)=0$.
With $N_0$ so defined, the lemma follows from the definition and
measurability of $ (\omega,\mathring{\omega}) \mapsto
\mathring{\mathcal{J}}_j^c(Y_{[0,t]}(\omega))(\mathring{\omega})$.
\end{proof}

\begin{lemma}
\label{lem:integral_representation}
Let $c$ be a positive integer and $j\in \{0,\ldots,n\}$. Then we have
$\check{P}$-almost surely that
\begin{equation*}
\int_{t_j}^{t_{j+1}} \langle Y_s, \,\mathrm{d}\mathring{\eta}_j^{c,c+1}(s)
\rangle
=
\mathring{\mathcal{J}_j^c}(Y_{[0,t]}).
\end{equation*}
\end{lemma}

\begin{proof}
Note that we can assume without loss of generality that $d_Y=1$ because the
result follows componentwise.
Then, let $K>0$ and $T = \inf\{s\in[0,t]\colon \lvert Y_s \rvert \leq K\}$
to define
\begin{equation*}
Y_s^K = Y_s \mathbb{I}_{s\leq T} + Y_T \mathbb{I}_{s> T}\; ;\qquad s\in[0,t].
\end{equation*}
Then Fubini's theorem and Lemma~\ref{lem:null_set} imply that
\begin{multline*}
\check{\mathrm{E}}
\Biggl[
\biggl(
\sum_{i=0}^{k-1}
    Y^K_{s_{i,j}}
    \Bigl(
    \mathring{\eta}_j^{c,c+1}(s_{i+1,j})
    -
    \mathring{\eta}_j^{c,c+1}(s_{i,j})
    \Bigr)
    -
    \mathring{\mathcal{J}_j^c}(Y^K_{[0,t]}(\omega))
\biggr)^2
\Biggr]\\
=
\int_{\Omega\setminus N_0}
\mathring{\mathrm{E}}
\bigl[
    \bigl(
    \mathring{\mathcal{J}}_j^{c,k}(Y^K_{[0,t]}(\omega))
    -
    \mathring{\mathcal{J}_j^c}(Y^K_{[0,t]}(\omega))
    \bigr)^2
\bigr]
\,\mathrm{d}\tilde{P}(\omega)
\end{multline*}
Now, since the function $s\mapsto Y^K_s(\omega)$ is continuous and
$\mathring{\mathcal{J}_j^c}(Y^K_{[0,t]}(\omega))$ is a version of the integral
$\int_{t_j}^{t_{j+1}}  Y^K_s(\omega) \,\mathrm{d}\mathring{\eta}_j^{c,c+1}(s)$
we have for every $\omega \in \Omega\setminus N_0$ that
\begin{equation*}
    \lim_{k\rightarrow \infty}
    \mathring{\mathrm{E}}
\bigl[
    \bigl(
    \mathring{\mathcal{J}}_j^{c,k}(Y^K_{[0,t]}(\omega))
    -
    \mathring{\mathcal{J}_j^c}(Y^K_{[0,t]}(\omega))
    \bigr)^2
\bigr]
= 0.
\end{equation*}
Moreover, clearly,
\begin{equation*}
    \mathring{\mathrm{E}}
\bigl[
    \bigl(
    \mathring{\mathcal{J}}_j^{c,k}(Y^K_{[0,t]}(\omega))
    -
    \mathring{\mathcal{J}_j^c}(Y^K_{[0,t]}(\omega))
    \bigr)^2
\bigr]
\leq
4K^2  \mathring{\mathrm{E}}[\mathring{\eta}_t^2] < \infty
\end{equation*}
So that we can conclude by the dominated convergence theorem that
\begin{multline*}
\lim_{k \rightarrow \infty}
\check{\mathrm{E}}
\Biggl[
\biggl(
\sum_{i=0}^{k-1}
    Y^K_{s_{i,j}}
    \Bigl(
    \eta_j^{c,c+1}(s_{i+1,j})
    -
    \eta_j^{c,c+1}(s_{i,j})
    \Bigr)
    -
    \mathring{\mathcal{J}_j^c}(Y^K_{[0,t]}(\omega))
\biggr)^2
\Biggr]\\
=
\int_{\Omega\setminus N_0}
\lim_{k\rightarrow\infty}
\mathring{\mathrm{E}}
\bigl[
    \bigl(
    \mathring{\mathcal{J}}_j^{c,k}(Y^K_{[0,t]}(\omega))
    -
    \mathring{\mathcal{J}_j^c}(Y^K_{[0,t]}(\omega))
    \bigr)^2
\bigr]
\,\mathrm{d}\tilde{P}(\omega) =0
\end{multline*}
As $K$ is arbitrary, the lemma is proved.
\end{proof}
Finally, we are ready to show the main result, Theorem~\ref{thm:robust}.
We restate it here again, in a slightly different manner which reflects the
current line of argument.

\begin{theorem}
\label{thm:final}
The random variable ${F}_{\varphi}^{\tau,m}(Y_{[0,t]})$ is a version of
$\pi_t^{\tau,m}(\varphi)$.
\end{theorem}
\begin{proof}
By the Kallianpur-Striebel formula it suffices to show that for all bounded and
Borel measurable functions $\varphi$ we have
$\tilde{P}$-almost surely
\begin{equation*}
    \rho_t^{\tau,m}(\varphi) = {G}_{\varphi}^{\tau,m}(Y_{[0,t]}).
\end{equation*}
Furthermore, this is equivalent to showing that for all continuous and bounded
functions $b\colon C([0,t];\mathrm{R}^{d_Y}) \to \mathrm{R}$ the equality
\begin{equation*}
    \tilde{\mathrm{E}}[\rho_t^{\tau,m}(\varphi)b(Y_{[0,t]})]
    =
    \tilde{\mathrm{E}}[{G}_{\varphi}^{\tau,m}(Y_{[0,t]})
        b(Y_{[0,t]})].
\end{equation*}
holds. As for the left-hand side we can write
\begin{align*}
    &\tilde{\mathrm{E}}[\rho_t^{\tau,m}(\varphi)b(Y_{[0,t]})]\\
    &=
    \tilde{\mathrm{E}}[\varphi(X_t)Z_t^{\tau,m}b(Y_{[0,t]})]\\
    &=
    \tilde{\mathrm{E}}[\varphi(X_t)\exp(\bar{\xi}_t^{\tau,m})b(Y_{[0,t]})]\\
    &=
    \check{\mathrm{E}}[\varphi(\mathring{X}_t)
    \exp(\mathring{\bar{\xi}}_t^{\tau,m})b(Y_{[0,t]})]\\
    &=
    \check{\mathrm{E}}[\varphi(\mathring{X}_t)
    \exp(\mathtt{IBP}(\mathring{\bar{\xi}}_t^{\tau,m}))b(Y_{[0,t]})]\\
\end{align*}
where $\mathtt{IBP}(\mathring{\bar{\xi}}_t^{\tau,m})$ is given by the
application of the integration by parts formula for semimartingales as
\begin{align*}
    \mathtt{IBP}(\mathring{\xi}_{t}^{\tau,m}) & =
        \sum_{j=0}^{n-1}\mathtt{IBP}(\mathring{\xi}_{t}^{\tau,m})(j)\\
        &=
        \sum_{j=0}^{n-1}
    \bigl\lbrace
    \mathring{\kappa}_j^{0,m}
    +
     \big\langle  \mathring{\eta}_j^{0,m}(t_{j+1}) ,{Y}_{t_{j+1}}  \big\rangle
     -
      \big\langle h(\mathring{X}_{t_{j}}) ,{Y}_{t_{j}}  \big\rangle
      -
    \int_{t_{j}}^{t_{j+1}}
        \big\langle {Y}_s , \mathrm{d}\mathring{\eta}_j^{0,m}(s) \big\rangle
        \bigr\rbrace\\
%%%
    \mathtt{IBP}(\mathring{\mu}^{\tau,m})\left(j\right)
    &=
    \mathring{\kappa}_j^{2,m}
    +
    \big\langle  \mathring{\eta}_j^{2,m}(t_{j+1}) ,{Y}_{t_{j+1}}  \big\rangle
    -
    \int_{t_{j}}^{t_{j+1}}
    \big\langle {Y}_s , \mathrm{d}\mathring{\eta}_j^{2,m}(s) \big\rangle\\
    \mathtt{IBP}(\mathring{\bar{\xi}}_{t}^{\tau,m})\left(j\right)&=
    \begin{cases}
        \displaystyle \mathtt{IBP}(\mathring{\xi}_{t}^{\tau,m})(j), &
        \mathrm{if}\ \ m=1,2\\
        \displaystyle \mathtt{IBP}(\mathring{\xi}_{t}^{\tau,m})(j)+
        \Gamma_{m,(t_{j+1}-t_j)}
            \left(\mathtt{IBP}(\mathring{\mu}^{\tau,m})\left(j\right)\right), &
            \mathrm{if}\ \ m>2
    \end{cases}\;.
\end{align*}
And, on the other hand, the right-hand side is
\begin{align*}
    &\tilde{\mathrm{E}}[{G}_{\varphi}^{\tau,m}(Y_{[0,t]})
        b(Y_{[0,t]})]\\
    &=
    \tilde{\mathrm{E}}[\varphi(X_t)\mathcal{Z}_t^{\tau,m}(Y_{[0,t]})
        b(Y_{[0,t]})]\\
    &=
    \tilde{\mathrm{E}}[\varphi(X_t)\exp(\bar{\Xi}_t^{\tau,m}(Y_{[0,t]}))
        b(Y_{[0,t]})]\\
    &=
    \tilde{\mathrm{E}}[\mathring{\mathrm{E}}[\varphi(\mathring{X}_t)
    \exp(\mathring{\bar{\Xi}}_t^{\tau,m}(Y_{[0,t]}))]
        b(Y_{[0,t]})]\\
    &=
    \check{\mathrm{E}}[\varphi(\mathring{X}_t)
    \exp(\mathring{\bar{\Xi}}_t^{\tau,m}(Y_{[0,t]}))
        b(Y_{[0,t]})],
\end{align*}
where the last equality follows from Fubini's theorem.
As the representations coincide, the theorem is thus proved.
\end{proof}


\newpage
\begin{thebibliography}{99.}%

\bibitem{bain2009fundamentals} Bain, A., Crisan, D.: Fundamentals of stochastic
filtering. Springer, Stochastic modelling and applied probability 60 (2009)
%
\bibitem{clark1978design} Clark, J.~M.~C.: The design of robust approximations
to the stochastic differential equations of nonlinear filtering. In J.K. Skwirzynski, editor,
Communication systems and random processes theory, volume 25 of Proc. 2nd NATO
Advanced Study Inst. Ser. E, Appl. Sci., pages 721-734. Sijthoff \& Noordhoff,
Alphen aan den Rijn (1978)
%
\bibitem{clark1980rate} Clark, J.M.C., Cameron, R.J.: The maximum rate of convergence of discrete approximations for stochastic differential equations.
In Stochastic differential systems, Lecture Notes in Control and Information Science 25, pages 162--171, Springer, Berlin-New York (1980)
%
\bibitem{crisan2005version} Clark, J., Crisan, D.: On a robust version of the integral representation formula of nonlinear filtering. In Probab. Theory Relat. Fields 133, 43–56 (2005). https://doi.org/10.1007/s00440-004-0412-5
%
\bibitem{crisan2019high} Crisan, D., Ortiz-Latorre, S.: A high order time
discretization of the solution of the non-linear filtering problem. Stoch PDE:
Anal Comp (2019) doi: 10.1007/s40072-019-00157-3
%
\bibitem{crisan2013robust} Crisan, D., Diehl, J., Friz, P. K., Oberhauser, H.: Robust filtering: Correlated noise and multidimensional observation.
In Ann. Appl. Probab. 23 (2013), no. 5, 2139--2160. doi:10.1214/12-AAP896.
%
\bibitem{davis1977linear} Davis, M.H.A.: Linear estimation and stochastic control. Chapman and Hall, London (1977)
%
\bibitem{davis1980computational} Davis, M.H.A., Wellings, P.H.: Computational problems in nonlinear filtering. In Analysis and Optimization of Systems, Springer (1980)
%
\bibitem{davis1981introduction} Davis, M.H.A., Marcus S.I.: An Introduction to Nonlinear Filtering.
In Hazewinkel M., Willems J.C. (eds) Stochastic Systems - The Mathematics of Filtering and Identification and Applications.
NATO Advanced Study Institutes Series (Series C - Mathematical and Physical Sciences), vol 78.
Springer, Dordrecht. doi: 10.1007/978-94-009-8546-9\_4 (1981)
%
\bibitem{davis1982pathwise} Davis, M.H.A.: A Pathwise Solution of the Equations of Nonlinear Filtering. In Theory of Probability and its Applications, 27:1, pages 167-175 (1982)
%
\bibitem{davis1987nonlinear} Davis, M.H.A., Spathopoulos, M. P.: Pathwise Nonlinear Filtering for Nondegenerate Diffusions with Noise Correlation. In
SIAM Journal on Control and Optimization, 25:2, pages 260-278 (1987)
%
\bibitem{davis1975nonlinear} Segall, A., Davis, M.H.A., Kailath, T.: Nonlinear filtering with counting observations.
In IEEE Transactions on Information Theory, volume 21, no. 2, pages 143-149.
doi: 10.1109/TIT.1975.1055360. (1975)
\end{thebibliography}
\end{document}